\documentclass[11pt]{amsart}

\usepackage{amssymb}
\usepackage{graphicx}
\usepackage{ifthen}
\usepackage{url}

\renewcommand{\vec}[1]{#1}
\newcommand{\E}{\mathbb E}
\newcommand{\Z}{\mathbb Z}
\newcommand{\R}{\mathbb R}
\newcommand{\C}{\mathbb C}
\newcommand{\MH}{{\mathcal H}}
\newcommand{\MNR}{{\mathcal R}}
\newcommand{\MP}{{\mathcal P}}
\newcommand{\MV}{{\mathcal V}}

\newcommand{\GL}{\mathsf{GL}}
\newcommand{\gldz}{\GL_d(\Z)}  
\newcommand{\gldr}{\GL_d(\R)}   
\newcommand{\sd}{{\mathcal S}^d}
\newcommand{\sdo}{{\mathcal S}^d_{>0}}

\DeclareMathOperator{\trace}{trace}
\DeclareMathOperator{\vol}{vol}
\DeclareMathOperator{\conv}{conv}
\DeclareMathOperator{\cone}{cone}
\DeclareMathOperator{\id}{id}
\DeclareMathOperator{\Min}{Min}
\DeclareMathOperator{\relint}{relint}
\DeclareMathOperator{\Aut}{Aut}
\DeclareMathOperator{\grad}{grad}

\newtheorem{definition}{Definition}
\newtheorem{theorem}[definition]{Theorem}

\newcounter{alg}

\newenvironment{bigalg}{\medskip%
                        \begin{figure}[htbp]%
                        \refstepcounter{alg}%
                        \begin{center}}%
                       {\end{center}\end{figure}\medskip}

\author{Achill Sch\"urmann}
\address{Achill Sch\"urmann, Mathematics Department, 
Otto-von-Guericke University of Magdeburg, 39106 Magdeburg, Germany}
\email{achill@math.uni-magdeburg.de}
\thanks{The author was supported by the 
Deutsche Forschungsgemeinschaft (DFG) under grant SCHU 1503/4-2.
He thanks the Hausdorff Research Institute for Mathematics for its
hospitality and support.}

\title{Enumerating perfect forms} 


\begin{document}

%\sloppy

% ============
% = Abstract =
% ============

\begin{abstract}
A positive definite quadratic form is called
perfect, if it is uniquely determined by its
arithmetical minimum and the integral vectors
attaining it. 
In this self-contained survey we explain how to
enumerate perfect forms in $d$ variables up to arithmetical
equivalence and scaling.
We put an emphasis on practical issues 
concerning computer assisted enumerations. 
For the necessary theory of Voronoi we provide
complete proofs based on Ryshkov polyhedra.
This allows a very natural generalization to
$T$-perfect forms, which are
perfect with respect to 
a linear subspace $T$ in the space of quadratic forms.
Important examples include Gaussian,
Eisenstein and Hurwitz quaternionic perfect forms, for which
we present new classification results in dimensions
$8,10$ and $12$.
\end{abstract}

\maketitle

%
% Introduction
%

\section{Introduction}

In this paper we are concerned with {\em perfect forms},
which are {\em real positive definite quadratic forms} 
\begin{equation}  \label{eqn:quadform}
Q[\vec{x}]
=
\sum_{i,j=1}^d q_{ij} x_i x_j
\end{equation}
in $d$ variables $\vec{x}=(x_1,\dots,x_d)^t\in\R^d$,
determined uniquely by their
{\em arithmetical minimum}
\begin{equation}  \label{eqn:arithmin}      
\lambda(Q) 
=
\min_{\vec{x}\in \Z^d\setminus\{\vec{0}\}} Q[\vec{x}]
\end{equation}
and {\em its representations} 
\begin{equation}  \label{eqn:arithrep} 
\Min Q =\{\vec{x}\in \Z^d : Q[\vec{x}] = \lambda(Q) \}.
\end{equation}

The study of perfect forms goes back to the 
work of Korkin and Zolotarev \cite{kz-1877}.
They observed that perfection is necessary 
for positive definite quadratic forms in order 
to give a local maximum of the {\em Hermite invariant}
\begin{equation}
\MH(Q)=\frac{\lambda(Q)}{(\det Q)^{1/d}}
.
\end{equation}
Such forms are called {\em extreme}.

As briefly reviewed in Section~\ref{sec:quadforms},
finding the global maximum of the Hermite invariant, 
or equivalently the densest lattice sphere packing
is a widely studied problem. In this article we describe 
the only known algorithmic
solution of this problem which works in principle in every dimension. 
It is based on the classification respectively enumeration
of perfect forms.
We refer to \cite{rb-1979}, \cite{cs-1998}, \cite{martinet-2003}, 
\cite{gruber-2007} and \cite{schuermann-2008} 
for further reading.

Based on perfect forms, 
Voronoi \cite{voronoi-1907} developed a 
polyhedral reduction theory, which was later found to have 
several applications in other contexts.
It has for example been used for compactification of moduli spaces
(cf. for example \cite{amrt-1975}, 
\cite{mcconnell-1998}, \cite{shepherd-barron-2006}),
for computing the cohomology of $\gldz$ and of congruence subgroups,
as well as for computing 
algebraic $K$-groups $K_d(\Z)$ for small $d$
and up to small torsion
(cf. \cite{soule-1999}, \cite{egs-2002} and the appendix in \cite{stein-2007}).
A basic task in these computations is the enumeration
of perfect forms. In some of the applications it is also necessary
to understand more of the structure of
the {\em Ryshkov polyhedron} 
(to be defined in Section~\ref{sec:ryshkov})
whose vertices are perfect forms.

In this article we explain Voronoi's 
theory based on the Ryshkov polyhedron.
We provide complete proofs for all 
of its required properties. We think that this 
view is more accessible than the usual dual 
viewpoint, originally taken by Voronoi and
by most other authors subsequently. 
Voronoi's algorithm can be simply described
as a traversal search on the graph
consisting of vertices and edges of the Ryshkov 
polyhedron. This viewpoint 
allows in particular a very simple and direct
generalization to so called $T$-perfect forms:
Intersecting a linear subspace $T$ with the
Ryshkov polyhedron yields a lower dimensional 
Ryshkov polyhedron whose vertices are $T$-perfect 
forms. Voronoi's theory immediately generalizes.

The article is organized as follows.
In Section~\ref{sec:quadforms} we review some
necessary background and notations.
In Section~\ref{sec:ryshkov} we define the 
Ryshkov polyhedron and prove that it is
``locally finite''. This yields
the grounds for Voronoi's algorithm to
be described in Section~\ref{sec:voronoi}.
Here we put special emphasis on practical 
issues related to running Voronoi's
algorithm on a computer.
In Section~\ref{sec:extremality}
we briefly explain how to determine extreme forms.
Section~\ref{sec:symmetries} contains some informations
on automorphism groups and their computation
and in Section~\ref{sec:Tperfect} we explain
the ``$T$-theory'', when 
restricting to a linear subspace~$T$.
As examples of linear subspaces that contain forms
invariant with respect to a finite group of automorphisms,
we consider in Section~\ref{sec:examples}
forms with a {\em Gaussian}, {\em Eisenstein} or 
{\em Hurwitz quaternionic} structure.
We obtain several new classification results.

\section{Background on positive definite quadratic forms}
\label{sec:quadforms}

In this section we review -- basically from scratch -- 
some of the historical background and notations
used in the remaining of the article.
The reader familiar with most of this background
may simply skip this section.

We consider real quadratic forms 
in $d$ variables as in \eqref{eqn:quadform},
hence with coefficients $q_{ij}\in\R$. 
By assuming $q_{ij}=q_{ji}$ without loss of generality, 
we simply identify the quadratic form $Q$
with the real symmetric matrix $Q=(q_{ij})_{i,j=1,\dots,d}$.
The space of all real quadratic forms in $d$ variables is 
identified with the space     \label{not:sd}
\[
\sd = \left\{ Q\in \R^{d\times d} : Q^t = Q \right\}
\]
of real symmetric $d\times d$ matrices.
Using matrix notation we have
$Q[\vec{x}] = \vec{x}^t Q \vec{x}$.
Endowed with the inner product
\[ 
\langle Q, Q'\rangle = \sum_{i,j=1}^d q_{ij}q'_{ij} = \trace (Q\cdot Q')
, 
\]
$\sd$ becomes a $\binom{d+1}{2}$-dimensional Euclidean space.

Two quadratic forms $Q,Q'\in\sd$ are
called {\em arithmetically (or integrally) equivalent},
if there exists a matrix $U$ in the group
$$  \label{not:gldz}
\gldz = \{ U \in \Z^{d\times d} : |\det U | = 1\}
$$ 
such that
$$
Q' = U^t Q U.
$$
Note that $Q[\Z^d]= Q'[\Z^d]$ for arithmetical
equivalent $Q$ and $Q'$, but the opposite may not hold.

A quadratic form $Q\in\sd$
is {\em positive definite},
if $Q[\vec{x}]>0$ for all $\vec{x}\in\R^d\setminus\{0\}$. The set of all
positive definite quadratic forms (PQFs from now on) is denoted by $\sdo$. 
It is not hard to see that $\sdo$ 
is an open (full dimensional) {\em convex cone}  
in $\sd$ with apex $0$.
In particular for $Q\in\sdo$, the {\em open ray}  
$\{ \lambda Q : \lambda>0 \}$   
is contained in $\sdo$ as well. 
Only for PQFs the arithmetical minimum defined in
\eqref{eqn:arithmin} is greater than~$0$.

A PQF $Q$ defines a real valued strictly convex function
on $\R^d$ and for $\lambda>0$
\begin{equation}  \label{eqn:ellipsoid}
E(Q,\lambda)
=
\{
\vec{x}\in\R^d : Q[\vec{x}] \leq \lambda
\}
\end{equation}
is a non-empty {\em ellipsoid}  \index{ellipsoid}
with center $\vec{0}$, providing a geometric interpretation of a PQF.
The arithmetical minimum is the smallest number $\lambda>0$ 
for which the ellipsoid $E(Q,\lambda)$ contains an integral
point aside of $\vec{0}$. 
The integral points $\vec{x}$ in $\Min Q$ (see \eqref{eqn:arithrep}) 
lie on the boundary of the ellipsoid $E(Q,\lambda(Q))$.

Hermite, who initiated the 
systematic  arithmetic study
of quadratic forms in $d$ variables
found in particular an upper bound
of the arithmetical minimum in terms of
the {\em determinant} $\det Q$ of $Q$:

\begin{theorem}[Hermite, \cite{hermite-1850}] \label{thm:hermite}
$$
\lambda(Q)\leq (\det Q)^{1/d} \cdot \left(\frac{4}{3}\right)^{(d-1)/2}
\;\;
\mbox{
for all $Q\in\sdo$.}
$$
\end{theorem}

Hermite's theorem implies
in particular the existence of {\em Hermite's constant}
\begin{equation} \label{eqn:hermite-constant}
\MH_d = \sup_{Q\in\sdo} \frac{\lambda(Q)}{(\det Q)^{1/d}}
.
\end{equation}
Hermite's constant and generalizations have been extensively studied,
e.g. in the context of algebraic number theory
and differential geometry. 
We refer to \cite{bavard-1997}, \cite{schmutz-1998},
\cite{coulangeon-2001} and \cite{watanabe-2004} for further reading.

The following {\em lattice sphere packing} interpretation is due to
Minkowski:
Using a {\em Cholesky decomposition} $Q=A^tA$ of a PQF $Q$,
with $A\in\gldr$,
the set $L=A\Z^d$ is a {\em (point) lattice}, that is,
a discrete subgroup of $\R^d$. The column vectors of the matrix~$A$
are referred to as a {\em basis} of~$L$. 
The maximum radius of non overlapping solid spheres
around lattice points of~$L$ is
$$        \label{not:lambda_L}
\lambda(L) = \frac{\sqrt{\lambda(Q)}}{2},
$$
the so called {\em packing radius}   
of $L$.
Denoting the solid {\em unit sphere} by $B^d$,
the {\em sphere packing density}  
$\delta(L)$ of a lattice $L$
is defined as the portion of space covered by solid spheres of radius $\lambda(L)$, hence 
$$     \label{not:sphere_packing_density}
\delta(L) = \frac{\vol(\lambda(L)B^d)}{\det L} =
\frac{\lambda(L)^d \vol B^d}{\det L}
.
$$
Note that $\delta$ is invariant with respect to isometries
and scalings of the lattice $L$.
The supremum
of possible lattice packing densities $\delta_d$ is, 
up to a constant factor,   
equal to a power of Hermite's constant.
Table~\ref{tab:sphere-packing-results} lists the dimensions
in which $\delta_d$ respectively Hermite's constant $\MH_d$ is known.

\begin{table}
\begin{tabular}{c|c|c|c|c}
$d$ & lattice & {$\delta_d$} & {$\MH_d$} & author(s) \\
\hline
 %$1$ & $\Z^1$               &  $1$        \\
 $2$ & ${\mathsf A}_2$    &  $0.9069\ldots$  & $\left(\frac{4}{3}\right)^{1/2}$ & Lagrange, 1773, \cite{lagrange-1773}\\
 $3$ & ${\mathsf A}_3={\mathsf D}_3$    &  $0.7404\ldots$  & $2^{1/3}$ &  Gauss, 1840, \cite{gauss-1840}\\
 $4$ & ${\mathsf D}_4$    &  $0.6168\ldots$  & $4^{1/4}$ & Korkin \& Zolotarev, 1877, \cite{kz-1877}\\
 $5$ & ${\mathsf D}_5$    &  $0.4652\ldots$  & $8^{1/5}$ & Korkin \& Zolotarev, 1877, \cite{kz-1877}\\
 $6$ & ${\mathsf E}_6$    &  $0.3729\ldots$  & $\left(\frac{64}{3}\right)^{1/6}$ & Blichfeldt, 1935, \cite{blichfeldt-1934} \\
 $7$ & ${\mathsf E}_7$    &  $0.2953\ldots$  & $64^{1/7}$ & Blichfeldt, 1935, \cite{blichfeldt-1934}\\
 $8$ & ${\mathsf E}_8$    &  $0.2536\ldots$  & $2$ & Blichfeldt, 1935, \cite{blichfeldt-1934}\\[0.3cm]
 $24$ & ${\mathsf \Lambda}_{24}$    &  $0.0019\ldots$  & $4$ & Cohn \& Kumar, 2004, \cite{ck-2004}\\  
\end{tabular}
\medskip
\caption{Known values of Hermite's constant.}
\label{tab:sphere-packing-results}
\end{table}

The lattices $\mathsf{A}_d$ for $d\geq 2$, 
$\mathsf{D}_d$ for $d\geq 3$ \label{not:root_lattices}
and $\mathsf{E}_d$ for $d=6,7,8$ are the so-called {\em root lattices}.
One of the most fascinating objects is the {\em Leech Lattice}~$\Lambda_{24}$
in $24$~dimensions.
Definitions and plenty of further information on these
fascinating lattices can be found in \cite{cs-1998}, \cite{martinet-2003}
and the online database \cite{ns-2008}.

Minkowski noticed \cite{minkowski-1891} that the trivial bound
\begin{equation} \label{eqn:trivial_bound}
\delta(L)\leq 1,
\end{equation}  
which is an immediate consequence of the sphere packing interpretation,
tremendously improves the upper bound for the arithmetical minimum
in Hermite's Theorem. In fact, \eqref{eqn:trivial_bound}
is equivalent to 
\begin{equation} \label{eqn:new-mink-bound}
\lambda(Q)\leq (\det Q)^{1/d} \cdot \frac{4}{(\vol B_d)^{2/d}}.
\end{equation}
showing that the exponential constant on the right in
Theorem~\ref{thm:hermite} can be replaced by 
a constant which grows roughly linear with~$d$.

This trivial, but significant improvement 
lead Minkowski to a powerful fundamental principle.
The ellipsoid $E(Q,r_Q)$, with $r_Q$ 
being the right hand side in \eqref{eqn:new-mink-bound},
has volume 
$$
\vol E(Q,r_Q)
= \vol(\sqrt{r_Q}A^{-1}B^d)
= r_Q^{d/2} (\det Q)^{-1/2} \vol B^d
= 2^d
.
$$
Minkowski discovered that  
not only ellipsoids of volume $2^d$ contain a
non-zero integral point, but also all other centrally symmetric
{\em convex bodies}  \index{convex!\hlp body} 
(non-empty, compact convex sets).

\begin{theorem}[Minkowski's Convex Body Theorem]
\label{thm:minkowski}
Any centrally symmetric convex body in $\R^d$
of volume $2^d$ contains a non-zero integral point.
\end{theorem}

\section{Ryshkov polyhedra}

\label{sec:ryshkov}

Since the Hermite invariant is invariant with respect to scaling,
a natural approach of maximizing it
is to consider all forms with a fixed 
arithmetical minimum, say~$1$, and minimize the determinant
among them.
We may even relax the condition on the arithmetical minimum
and only require that it is at least~$1$.
In other words, we have
$$
\MH_d
=
1 / \inf_{\MNR} (\det Q)^{1/d}
,
$$
where  
\begin{equation}  \label{eqn:ryshkov-polyhedron}
\MNR =
\left\{
Q\in\sdo : \lambda(Q)\geq 1
\right\}
.
\end{equation}
We refer to $\MNR$ as {\em Ryshkov polyhedron},
as it was Ryshkov \cite{ryshkov-1970}
who noticed that this view
on Hermite's constant allows a simplified 
description of Voronoi's theory.

Because of the fundamental identity
\[
Q[\vec{x}] = \langle Q,\vec{x}\vec{x}^t \rangle 
,
\]
quadratic forms $Q\in \sd$ attaining a fixed value 
on a given $\vec{x}\in\R^d\setminus\{\vec{0}\}$ 
lie all in a {\em hyperplane}
({\em affine subspace of codimension~$1$}).
Thus Ryshkov polyhedra $\MNR$ are intersections of 
infinitely many {\em halfspaces}:   
\begin{equation}  \label{eqn:Plambda}
\MNR = 
\{
Q\in\sdo : \langle Q, \vec{x}\vec{x}^t \rangle \geq \lambda 
\mbox{ for all } \vec{x}\in\Z^d\setminus\{\vec{0}\}
\}
.
\end{equation}

We show below that are $\MNR$ is ``locally like a polyhedron''.
Its {\em vertices} are precisely 
the perfect forms with arithmetical minimum $1$.

{\em Background on polyhedra}.
Before we give the precise statement, and for later purposes, 
we need some basic notions from the theory of polyhedra.
As general references for further reading we recommend
the books \cite{ms-1971}, \cite{shrijver-1986}, 
\cite{ziegler-1998}, \cite{gruenbaum-2002}.
A {\em convex polyhedron}  
$\MP\subseteq \E$ in a {\em Euclidean space} $\E$
with inner product $\langle \cdot , \cdot \rangle$ 
(e.g. $\E= \sd$) 
can be defined by a finite set of linear inequalities
({\em $\mathcal{H}$-description})
$$
\MP=\{ \vec{x}\in \E : \langle \vec{a}_i, \vec{x} \rangle \geq b_i , i=1,\ldots,m\},
$$
with $\vec{a}_i\in \E$ and $b_i\in\R$ for $i=1,\ldots,m$.
If the number of inequalities $m$ in the description 
is minimum, we say it is non-redundant. 
The dimension $\dim \MP$ of $\MP$ is the dimension of the smallest
affine subspace containing it. 
Under the assumption that $\MP$ is full-dimensional 
every inequality $i$ of a non-redundant description defines 
a {\em facet}  
$\{\vec{x}\in \MP : \langle \vec{a}_i, \vec{x} \rangle = b_i\}$ of $\MP$, 
which is a $(d-1)$-dimensional convex polyhedron contained
in the boundary of $\MP$. More generally, an intersection of a hyperplane
with the boundary of $\MP$ is called a {\em face} of $\MP$, if $\MP$ is contained
in one of the two halfspaces bounded by the hyperplane.
The faces are polyhedra themselves; faces of dimension~$0$
and dimension~$1$ are called vertices and edges.

By the Farkas-Minkowski-Weyl Theorem
(see e.g.\ \cite[Corollary 7.1a]{shrijver-1986}),
$\MP$ can also be described by a finite set of generators
({\em {$\mathcal{V}$}-description}):
\begin{eqnarray*}
\MP & = & \conv\{\vec{v}_1,\ldots,\vec{v}_k\} + \cone\{\vec{v}_{k+1},\dots,\vec{v}_n\}  
\\
  & = & \{ \sum_{i=1}^n \lambda_i \vec{v}_i : \lambda_i\geq 0, \sum_{i=1}^k \lambda_i=1\} 
\end{eqnarray*}
where $\vec{v}_i\in \E$ for $i=1,\ldots,n$.
Here $\conv M$ denotes the {\em convex hull} 
and $\cone M$ the {\em conic hull} of a set $M$.
If the number of generators is minimum, the description is again called 
{\em non-redundant}. 
In the non-redundant case, the
generators $\vec{v}_i$, $i=1,\dots,k$, are called {\em vertices}  
and $\R_{\geq 0}\vec{v}_i$, $i=k+1,\dots,n$, are the {\em extreme rays}   
of $\MP$. In case $\MP$ is bounded we have $n=k$ and we speak of 
a {\em convex polytope}.

There exist several different approaches and corresponding 
software for the fundamental task of converting $\mathcal{H}$-descriptions
of polyhedra into $\mathcal{V}$-descriptions and vice versa
(see for example {\tt cdd} \cite{cdd} and {\tt lrs} \cite{lrs}).

{\em Locally finite polyhedra}.
We say that an intersection of infinitely many halfspaces,
$\MP = \bigcap_{i=1}^\infty H^+_i$,
is a {\em locally finite polyhedron}, if the intersection
with an arbitrary polytope is a polytope. So, locally $\MP$
``looks like a polytope''.

\begin{theorem} \label{thm:locally-finite-polyhedron}
For $d\geq 1$, the Ryshkov polyhedron $\MNR$ 
(see \eqref{eqn:Plambda}) is a locally finite polyhedron.
\end{theorem}

\begin{proof}
By applying Minkowski's convex body Theorem \ref{thm:minkowski},
we show below that 
\begin{equation} \label{eqn:PlambdaC} 
\MNR \cap \{ Q\in\sd : \trace Q \leq C \}
\end{equation}
is a polytope (possibly the empty set) for every constant $C$. 
This proves the theorem, since
$$\trace Q = \langle Q, \id_d \rangle \leq C$$ 
determines a halfspace containing a bounded section of $\sdo$.

The sets \eqref{eqn:PlambdaC} are polytopes if 
the set of all $\vec{x}\in\Z^d\setminus\{\vec{0}\}$
with $Q[\vec{x}]=1$ (or $Q[\vec{x}]\leq 1$)
for some forms $Q$ in \eqref{eqn:PlambdaC} is finite.
We show below that the absolute value of coordinates 
$$
m=\max_{i=1,\dots, d} | x_i |
$$
of $\vec{x}$ with this property is bounded.

Let $Q$ be a PQF in \eqref{eqn:PlambdaC}.
Then the ellipsoid 
$E(Q, 1)
=\{ \vec{x}\in\R^d : Q[\vec{x}] \leq 1 \}$
does not contain any point of $\Z^d\setminus\{\vec{0}\}$
in its interior. So in particular $\vol E(Q,1) \leq 2^d$
by Minkowski's convex body theorem.
Since
$$
1 \leq Q[\vec{e}_i] 
\leq 
(\trace Q )
-\sum_{j\not= i} Q[\vec{e}_j]
\leq
C - (d-1) 
,
$$
we know that $E(Q,1)$
contains the {\em cross polytope}  \index{cross polytope}
\begin{equation} \label{eqn:crosspoly}
C' \cdot \conv \{ \pm \vec{e}_i 
\; : \; i=1,\dots,d \}
\end{equation}
with 
$$C'=\left(C - (d-1)\right)^{-1/2}.$$

For $\vec{x}$ with $Q[\vec{x}]\leq 1$ consider the polytope
defined as the convex hull 
of $\pm \vec{x}$ and the cross polytope \eqref{eqn:crosspoly}.
It is contained in $E(Q,1)$. On the other hand, this polytope
contains the convex hull $P$ of $\pm \vec{x}$ 
and the $(d-1)$-dimensional cross polytope
$$C' \cdot \conv \{ \pm \vec{e}_i : i=1,\dots,d, i\not=j \},$$
where $j\in\{1,\dots,n\}$ is chosen such that $|x_j|$ attains $m$.
Thus setting $C''$ to be the $(d-1)$-dimensional volume 
of latter $(d-1)$-dimensional cross polytope
we get
$$
m \cdot \frac{2}{d} C'' =
\vol P \leq 
\vol \left( \conv\{\pm \vec{x}, \mbox{\eqref{eqn:crosspoly}}\} \right)
< \vol E(Q,1)\leq 2^d
.
$$
Hence we obtain the desired bound on~$m$ (depending only on $d$).
\end{proof}

One consequence of the Theorem is the fact that
Hermite's constant can only be attained by perfect forms,
which was first observed 
by Korkin and Zolotarev in \cite{kz-1877}.
This follows immediately from the 
following Theorem.

\begin{theorem}[Minkowski~\cite{minkowski-1905}]  \label{thm:concave-det}
$(\det Q)^{1/d}$ is a strictly concave function on~$\sdo$.
\end{theorem}

For a proof see for example \cite{gl-1987}.
Note, that in contrast to 
$(\det Q)^{1/d}$, the function $\det Q$ is not 
a concave function on~$\sdo$ (cf. \cite{nelson-1974}).
However Minkowski's theorem implies that the set
\begin{equation} \label{eqn:det-greater-equal-D}
\{Q\in \sdo : \det Q \geq D \}
\end{equation}
is strictly convex for $D>0$.

{\em Finiteness up to equivalence}.
The operation of $\gldz$ on $\sdo$ leaves $\lambda(Q)$, $\Min Q$
and also $\MNR$ invariant.
$\gldz$ acts on the sets of faces of a given dimension,
thus in particular on the sets of vertices, edges and facets of 
$\MNR$.
The following theorem shows that 
the Ryshkov polyhedron $\MNR$ contains only finitely 
many arithmetically inequivalent vertices.
By Theorem \ref{thm:concave-det} this implies in particular 
that $\MH_d$ is actually attained, namely by some perfect forms.

\begin{theorem}[Voronoi 1907] \label{thm:voronoi}
Up to arithmetical equivalence and scaling there exist
only finitely many perfect forms in a given dimension $d\geq 1$.
\end{theorem}

\begin{proof}
In the proof of Theorem \ref{thm:locally-finite-polyhedron}
we showed that the set \eqref{eqn:PlambdaC} of PQFs $Q$
with $\lambda(Q)\geq 1$ and $\trace Q\leq C$
is a polytope, hence has only finitely many vertices.
Therefore it suffices to show that every
perfect PQF $Q$ with $\lambda(Q)=1$ (a vertex of the Ryshkov polyhedron $\MNR$)
is arithmetically equivalent to a form with trace smaller 
than some constant depending only on the dimension~$d$.

By Hermite's Theorem~\ref{thm:hermite}
we find an equivalent PQF $Q'$ with
\begin{equation} \label{eqn:prod-qii_prime}
\prod_{i=1}^d q'_{ii} 
\leq 
\left(\frac{4}{3}\right)^{d(d-1)/2} \cdot \det Q'
.
\end{equation}
The determinant $\det Q'=\det Q$ can be bounded by
$1$ because of {\em Hadamard's inequality}
showing
\begin{equation}  \label{eqn:hadamard}
\det Q \leq Q[\vec{a}_1]\cdots Q[\vec{a}_d]
\end{equation}
for $Q\in\sdo$ and linearly independent $\vec{a}_1,\dots, \vec{a}_d\in \Z^d$.
Latter applies in particular to linearly independent vectors in $\Min Q$, 
respectively $\Min Q'$.
The existence of $d$ linear independent vectors in $\Min Q$ for a perfect
form $Q$ follows from the observation that the rank-$1$ forms
$\vec{x}\vec{x}^t$ with $\vec{x}\in\Min Q$ have to span $\sd$,
since they uniquely determine $Q$ through the linear 
equations $\langle Q, \vec{x}\vec{x}^t \rangle = \lambda(Q)$.
If however $\Min Q$ does not span $\R^d$ then these rank-$1$
forms can maximally span a $\binom{d}{2}$-dimensional subspace of $\sd$.

Because of $q'_{ii}\geq 1$ we find 
$$
q'_{kk}\leq \prod_{i=1}^d q'_{ii} \leq \left( \frac{4}{3} \right)^{d(d-1)/2} .
$$
From this we obtain the desired upper bound for the trace of $Q'$:
$$
\trace Q' = \sum_{k=1}^d q'_{kk} 
\leq d \left(\frac{4}{3}\right)^{d(d-1)/2} 
.
$$
\end{proof}

\section{Voronoi's algorithm}

\label{sec:voronoi}

The vertices (perfect PQFs) and edges of $\MNR$   
form the (abstract) {\em Voronoi graph in dimension $d$}.
Two vertices, respectively perfect PQFs $Q$ and $Q'$ 
are connected by an edge if the line segment $\conv\{Q,Q'\}$ 
is an edge of~$\MNR$. In this case we say that $Q$
and $Q'$ are {\em contiguous perfect forms}  
(or {\em Voronoi neighbors}).  \index{Voronoi!\hlp neighbors}
By Theorem \ref{thm:voronoi}, for given $d$, there are only finitely many
vertices (and edges) of the Voronoi graph up to arithmetical equivalence.
Therefore, one can enumerate perfect PQFs (up to arithmetical equivalence and scaling)
by a {\em graph traversal algorithm},  
which is known as 
{\em Voronoi's algorithm} 
(see Algorithm \ref{alg:voronoi-algorithm}).

\begin{bigalg}   \label{alg:voronoi-algorithm}
\fbox{
\begin{minipage}{12.0cm}
\begin{flushleft}
\smallskip
\textbf{Input:} Dimension $d$.\\
\textbf{Output:} A complete list of inequivalent perfect forms in $\sdo$.\\
\smallskip
Start with a perfect form $Q$.\\
\begin{enumerate}
\item[1.]
Compute $\Min Q$ and describing inequalities of polyhedral cone
\begin{equation} \label{eqn:Qcone}
\MP(Q)=\{ Q' \in \sd \; : \; Q'[\vec{x}]\geq 0 \mbox{ for all } \vec{x}\in\Min Q \}
\end{equation}

\item[2.]
Enumerate extreme rays $R_1,\dots, R_k$ of the cone $\MP(Q)$

\item[3.]
Determine contiguous perfect forms
$Q_i=Q+\alpha R_i$, $i=1,\dots,k$

\item[4.]
Test if $Q_i$ is arithmetically equivalent to a known form

\item[5.]
Repeat steps 1.--4. for new perfect forms
\end{enumerate}
\end{flushleft}
\end{minipage}
}
\\[1ex]
{\sc Algorithm \arabic{alg}.} Voronoi's algorithm.
\end{bigalg}

As an initial perfect form we may for example choose
{\em Voronoi's first perfect form},
\index{Voronoi!\hlp's first perfect form}
\index{perfect form!Voronoi's first \hlp}
which is associated to the {\em root lattice $\mathsf{A}_d$}.   \index{root lattice}
For example take $Q_{\mathsf{A}_d}=(q_{i,j})_{1\leq i,j\leq d}$
with $q_{i,i}=2$, $q_{i,i-1}=q_{i-1,i}=-1$ and $q_{i,j}=0$ otherwise
(see \cite[Section~6.1]{cs-1998} or \cite[Section~4.2]{martinet-2003}).

One key ingredient, not only for step~1., 
is the computation of representations of the arithmetical minimum.
For it we may use the {\em Algorithm of Fincke and Pohst}
\index{algorithm!\hlp of Fincke and Pohst}
\index{Fincke and Pohst algorithm}
(cf. \cite{cohen-1993}):
Given a PQF $Q$, it allows to compute all $\vec{x}\in\Z^d$ with
$Q[\vec{x}]\leq C$ for some constant $C>0$. 
For $C=\min_{i=1,\dots,d} q_{ii}$ a 
non-zero integral vector $\vec{x}$ with $Q[\vec{x}]\leq C$ exists,
hence in particular $\lambda(Q)\leq C$.  
The Fincke and Pohst algorithm makes use of 
the {\em Lagrange expansion}  
of $Q$, given by
\begin{equation} \label{eqn:lagrange_expansion}
Q[\vec{x}] = \sum_{i=1}^d A_i \left( x_i -\sum_{j=i+1}^d \alpha_{ij} x_j\right)^2,
\end{equation}
with unique positive {\em outer coefficients}
\index{outer!\hlp coefficient}  \index{coefficient!outer \hlp}
$A_i$ and {\em inner coefficients}
\index{inner coefficient}  \index{coefficient!inner \hlp}
$\alpha_{ij}\in\R$, for
$i=1,\dots,d$ and $j=i+1,\dots,d$.
By it, it is possible to restrict the search to 
integral vectors $\vec{x}$ with 
$$
\left|
x_i - \sum_{i=1}^d \alpha_{ij} x_j
\right|
\leq \sqrt{\frac{C}{A_i}}
$$
for $i=d,\dots,1$. Here, the bound on the coordinate $x_i$
depends on fixed values of $x_{i+1},\dots, x_d$, for which we have 
only finitely many possible choices.
Implementations are provided in 
computer algebra systems like \texttt{Magma} \cite{magma} 
or \texttt{GAP} \cite{gap}
(see also \texttt{shvec} by Vallentin \cite{shvec-1999}).

For step~2., observe that  
the homogeneous cone \eqref{eqn:Qcone} is a translate of the 
{\em support cone}  
$$
\{ Q' \in \sd \; : \; Q'[\vec{x}]\geq Q[\vec{x}] \mbox{ for all } \vec{x}\in\Min Q \}
$$
of $Q$ at $\MNR$.
Having its {$\mathcal{H}$}-description
(by linear inequalities)
we can transform it to its {$\mathcal{V}$}-description
and obtain its extreme rays.
The extreme rays $R$ provided by $Q$ through \eqref{eqn:Qcone} 
are easily seen to be indefinite quadratic forms 
(see \cite{martinet-2003}).

In step~3.,
the contiguous perfect forms (Voronoi neighbors) 
of $Q$ are of the form $Q + \rho R$, where $\rho$ is the
smallest positive number such that $\lambda(Q + \rho R) = \lambda$ and 
$\Min(Q + \rho R) \not \subseteq \Min Q$. 
It is possible to determine $\rho$, for example with
Algorithm \ref{alg:determination-of-voronoi-neighbors}:

\begin{bigalg}   \label{alg:determination-of-voronoi-neighbors}
\fbox{
\begin{minipage}{12.0cm}
\begin{flushleft}
\smallskip
\textbf{Input:} A perfect form $Q\in\sdo$ and an extreme ray $R$ of \eqref{eqn:Qcone}\\
\textbf{Output:} $\rho>0$ with $\lambda(Q + \rho R) = \lambda(Q)$ and 
$\Min(Q + \rho R) \not \subseteq \Min Q$.\\
\smallskip
$(l, u) \leftarrow (0,1)$\\
\smallskip
\textbf{while} $Q + u R \not\in \sdo$ or $\lambda(Q + u R) = \lambda(Q)$ \textbf{do}\\
\hspace{2ex} \textbf{if} $Q + u R \not\in \sdo$ \textbf{then} $u \leftarrow (l + u)/2$\\
\hspace{2ex} \textbf{else} $(l,u) \leftarrow (u, 2u)$\\
\hspace{2ex} \textbf{end if}\\
\textbf{end while}\\
\smallskip
\textbf{while} $\Min(Q + l R) \subseteq \Min Q$ \textbf{do}\\
\hspace{2ex} $\gamma \leftarrow \frac{l + u}{2}$\\
\hspace{2ex} \textbf{if} $\lambda(Q + \gamma R) \geq \lambda(Q)$ \textbf{then} $l \leftarrow \gamma$\\
\hspace{2ex} \textbf{else} \\
\hspace{5ex}$u \leftarrow \min \left\{ (\lambda(Q)-Q[\vec{v}])/R[\vec{v}] : 
                             \vec{v}\in \Min (Q+\gamma R), R[\vec{v}] < 0 \right\} \cup \{ \gamma \} $\\
\hspace{2ex} \textbf{end if}\\
\textbf{end while}\\
\smallskip
$\rho \leftarrow l$
\end{flushleft}
\end{minipage}
}
\\[1ex]
{\sc Algorithm \arabic{alg}.} Determination of Voronoi neighbors.
\end{bigalg}

In phase I (first {\tt while} loop), the procedure determines lower and upper bounds $l$ and $u$ for
the desired value $\rho$, such that $Q+ l R, Q + u R \in \sdo$ with
$\lambda(Q+ l R)=\lambda$ and $\lambda(Q + u R)<\lambda$.
In phase II, the value of $\rho$ is determined. 
Note that replacing the assignment of $u$ by
the simpler assignment
$u \leftarrow \gamma$ corresponds to a binary search coming
at least arbitrarily close to $\rho$. However, it may never 
reach the exact value.

For step~4. observe,
that based on an algorithm to compute short vectors (for example the one by Fincke-Pohst described above), 
it is possible to test algorithmically 
if two PQFs $Q$ and $Q'$ are arithmetically
equivalent. That is, because the existence of $U\in\gldz$
with $Q'=U^t Q U$ implies
$$q'_{ii}= Q'[\vec{e}_i] = Q[\vec{u}_i].$$
Hence for the $i$-th column $\vec{u}_i$ of $U$ we have only finitely 
many choices. This idea, 
but more sophisticated, is implemented in
\texttt{isom} by Plesken and Souvignier \cite{ps-1997}, which is also
part of \texttt{Magma} \cite{magma} and \texttt{Carat} \cite{carat}.
Note that isometry tests for perfect forms can be simplified,
because it suffices to find a $U\in\gldz$ with 
$U\Min Q' = \Min Q$.

Using the described software tools it is possible to
verify the results of Table~\ref{tab:num-perf-forms} below on any standard PC up to dimension~$6$.
Note however, that this computation was already done without a computer by Barnes \cite{barnes-1957a}.
In dimension~$7$ and beyond the explained procedure has a seemingly insuperable ``bottleneck'': 
The enumeration of extreme rays for support cones with many facets,
respectively for perfect forms with large sets $\Min Q$.

\begin{table} 
\begin{tabular}{c|c|c|c}
$d$ & \# perf. forms & \# ext. forms & author(s)\\
\hline 
$2$ & $1$ & $1$ & Lagrange, 1773, \cite{lagrange-1773} \\
$3$ & $1$ & $1$ & Gau\ss, 1840, \cite{gauss-1840}\\
$4$ & $2$ & $2$ & Korkin \& Zolotarev, 1877, \cite{kz-1877}\\
$5$ & $3$ & $3$ & Korkin \& Zolotarev, 1877, \cite{kz-1877}\\
$6$ & $7$ & $6$ & Barnes, 1957, \cite{barnes-1957a}\\
$7$ & $33$ & $30$ & Jaquet-Chiffelle, 1993, \cite{jaquet-1993}\\[0.1cm]
$8$ & $10916$ & $2408$ & Dutour Sikiri\'c, Sch\"urmann \& Vallentin\\  
                         
$9$ & $> 500000$ &  &    2005, \cite{latgeo-2005},\cite{dsv-2007b}, cf. \cite{riener-2006}\\
\end{tabular}
\medskip
\caption{Known numbers of perfect and extreme forms.}
\label{tab:num-perf-forms}
\end{table}

There have been several attempts of using computers to (try to) enumerate
perfect forms. Larmouth \cite{larmouth-1971} was the first who 
implemented it and was able to verify the result of Barnes \cite{barnes-1957a} 
up to dimension $6$. Also, Stacey \cite{stacey-1975} and Conway and Sloane \cite{cs-1988a} 
used computer assistance for their attempts to classify the perfect forms in dimension $7$.
Exploiting symmetries, Jaquet-Chiffelle \cite{jaquet-1993} was able
to enumerate all perfect forms in dimension $7$.
Recently, together with Mathieu Dutour Sikiri\'c and Frank Vallentin
we were able to finish the classification in dimension~$8$
(see \cite{latgeo-2005} and \cite{dsv-2007b}).

\section{Eutaxy and Exremality}

\label{sec:extremality}

Not every perfect form is extreme, hence gives a local maximum of
the Hermite invariant, as shown in Table~\ref{tab:num-perf-forms} from
dimension~$6$ onwards.

In order to characterize extreme forms
the notion of {\em eutaxy} is used:
A PQF $Q$ is called {\em eutactic},   
\index{quadratic form!eutactic \hlp}
\index{eutactic!\hlp quadratic form}
if its inverse $Q^{-1}$ is contained in the (relative) interior 
$\relint \MV(Q)$ of its {\em Voronoi domain}
$$
\MV(Q)
= \cone \{ \vec{x}\vec{x}^t : \vec{x}\in\Min Q \}
.
$$
Note that the Voronoi domain is full-dimensional 
if and only if $Q$ is perfect.
Note also that the rank-$1$ forms $\vec{x}\vec{x}^t$ give
inequalities $\langle Q, \vec{x}\vec{x}^t \rangle \geq 1$ 
defining the Ryshkov polyhedron and by this the 
Voronoi domain of $Q$ is equal to the {\em normal cone}
\begin{equation} \label{eqn:normal-cone}
\{N \in \sd : \langle N, Q/\lambda(Q) \rangle \leq \langle N , Q' \rangle \mbox{ for all } Q'\in \MNR \}
\end{equation}
of $\MNR$ at $Q/\lambda(Q)$.

Algebraically the eutaxy condition $Q^{-1} \in \relint \MV(Q)$
is equivalent to the existence
of positive $\alpha_{\vec{x}}$ with 
\begin{equation} \label{eqn:eutaxy-algebraic}
Q^{-1} = \sum_{\vec{x}\in\Min Q} \alpha_{\vec{x}} \vec{x}\vec{x}^t
.
\end{equation} 
Computationally, eutaxy of $Q$ can be tested by solving the {\em linear program}  \index{linear!\hlp program}
\begin{equation} \label{eqn:eutaxy-lp}
\max \alpha_{\min}
\quad \mbox{s.t. $\alpha_{\vec{x}}\geq \alpha_{\min}$ and \eqref{eqn:eutaxy-algebraic} holds.}
\end{equation}
The form $Q$ is eutactic, if and only if 
the maximum is greater $0$.

Voronoi \cite{voronoi-1907} showed that perfectness, together with
eutaxy implies extremality and vice versa.
(Eutaxy alone does not suffice for extremality.)
By solving the linear program~\eqref{eqn:eutaxy-lp}
for perfect forms a list of extreme forms
can be obtained.  
This was done by Riener \cite{riener-2006} for the
$8$~dimensional perfect forms, showing that only 
$2408$ of them are extreme (see Table~\ref{tab:num-perf-forms}).

Geometrically the characterization of extreme forms by Voronoi
can easily be seen from the identity
\begin{equation}  \label{eqn:gradient-det}
\grad \det Q = (\det Q)Q^{-1}
\end{equation}
for the gradient of $\det Q$. By it, the
tangent hyperplane $T$ in $Q$
of the smooth {\em determinant-$\det Q$-surface}  
$$
S = 
\{
Q'\in\sdo : \det Q' = \det Q
\}
$$
is given by
$$
T = 
\{
Q'\in\sd :
\langle Q^{-1},Q'\rangle
=
\langle Q^{-1},Q \rangle
\}
.
$$

Or in other words,
$Q^{-1}$ is a normal vector of 
the tangent plane $T$
of $S$ at $Q$. By Theorem \ref{thm:concave-det}  
the surface $S$ is contained in the halfspace
\begin{equation} \label{eqn:tangent-plane}
\{Q'\in\sd : \langle Q^{-1} , Q'-Q \rangle \geq 0 \}
,
\end{equation}
with $Q$ being the unique intersection point of $S$ and $T$.

As a consequence,
a perfect form $Q$ with $\lambda(Q)=1$ attains a local minimum of $\det Q$
(hence is extreme) if and only if the halfspace \eqref{eqn:tangent-plane}
contains the Ryshkov polyhedron $\MNR$, and its boundary meets $\MNR$
only in $Q$. This is easily seen to be equivalent to the condition that
the normal cone (Voronoi domain) $\MV(Q)$ of $\MNR$ at $Q$
contains $Q^{-1}$ in its interior.

\section{Automorphism groups} 
\label{sec:symmetries}

The recent enumeration success in dimension~$8$ was previously not possible,
because the computation of extreme rays was in particular 
difficult for the support cones associated to the highly symmetric forms
associated to the root lattices $\mathsf{E}_7$ and $\mathsf{E}_8$. 
Note that the enumeration of extreme rays is a known difficulty 
in many problems, for example in combinatorial optimization.
Martinet stated that 
``it seems plainly impossible to classify $8$-dimensional perfect lattices''
(see \cite[p.218]{martinet-2003}).
However, it is possible to overcome these difficulties 
to some extend by exploiting symmetries in the computation.
For a survey on such symmetries exploiting techniques
we refer to \cite{bds-2007}.

In general the  {\em automorphism group}  
(or {\em symmetry group}) 
of a quadratic form $Q\in \sd$, is defined by
$$   \label{not:autom_of_form}
\Aut Q = \{ U \in \gldz : U^t Q U = Q \}
.
$$
As in the case of arithmetical equivalence, 
we can determine $\Aut Q$, based on the knowledge of all vectors
$\vec{u}\in\Z^d$ with $Q[\vec{u}]=q_{ii}$ for some $i\in\{1,\dots,d\}$.
Again, {\tt Magma} \cite{magma}, based on an implementation of 
Plesken and Souvignier (also available in {\tt Carat} \cite{carat}), 
provides a function for this task.

For $Q\in\sdo$ with $\lambda(Q)=1$, the support cone $\MP(Q)$ at $Q$ 
of the Ryshkov polyhedron $\MNR$ (see \eqref{eqn:Qcone}) and its dual,
the Voronoi domain $\MV(Q)$, inherit every symmetry of $Q$.
That is, for all $U\in\Aut Q$ we have 
$$
U^t \MP(Q) U = \MP(Q)
\qquad
\mbox{and}
\qquad
U^t \MV(Q) U = \MV(Q)
.
$$

The automorphism group of a PQF $Q$ is always finite. 
On the other hand, for every finite subgroup $G$ of $\gldz$,
there exists a PQF $Q$ with $G\subseteq \Aut Q$.
For example, given an arbitrary $Q'\in \sdo$, the PQF
$$
Q=\sum_{U\in G} U^t Q' U 
$$
is invariant with respect to $G$, hence satisfies
$G \subseteq \Aut Q$.

For a finite group $G\subset \gldz$, 
the {\em space of invariant quadratic forms}
\index{space of invariant quadratic forms}
\begin{equation} \label{eqn:space-of-invariant-forms}
T_G = \left\{ Q\in\sd : U^t Q U = Q \mbox{ for all } U \in G \right\}
\end{equation}
is a linear subspace of $\sd$;
$T_G\cap \sdo$ is called {\em Bravais manifold}  \index{Bravais!\hlp manifold} 
of $G$.

\section{$T$-perfect forms}

\label{sec:Tperfect}

Since the enumeration of all perfect forms 
becomes practically impossible in higher dimensions
(due to the complexity of the Ryshkov polyhedron $\MNR$),
it is natural to restrict classifications to 
certain Bravais manifolds. This is in particular
motivated by the fact that all forms known to attain
the Hermite constant have large symmetry groups.

Within $T_G$ we are lead to the theory of
{\em $G$-perfect forms}  
of Berg\'e, Martinet and Sigrist \cite{bms-1992}.
It generalizes to a theory of {\em $T$-perfect forms},  
where $T\subseteq \sd$ is some linear subspace (see \cite{martinet-2003}).
Suitable linear subspaces $T$ allow systematic 
treatments of important classes of forms.
Examples are {\em Eisenstein}, {\em Gaussian} and 
{\em Hurwitz quaternionic forms} as explained in Section~\ref{sec:examples}.
For further informations on 
classes as {\em cyclotomic forms} or forms having a
fixed section we refer to \cite{sigrist-2000} and \cite{martinet-2003}.

Our viewpoint developed in this article (based on Ryshkov polyhedra) 
allows a straightforward description of the ``$T$-theory''.
Given a linear subspace $T\subseteq \sd$ we simply consider 
the intersection 
\begin{equation}  \label{eqn:locally-finite-T-sections}
\MNR \cap T
.
\end{equation}
It is again a locally finite polyhedron which we call
a Ryshkov polyhedron too.
Its vertices are called {\em $T$-perfect forms}. 
In case $T=T_G$, where $G$ is a finite subgroup 
we speak of {\em $G$-perfect forms}. 
One should be aware that in general, 
$T$-perfectness does not imply perfectness.

We have to modify the notion of equivalence.
Two PQFs $Q$ and $Q'$ are called 
{\em $T$-equivalent}  
if there exists a $U\in \gldz$ with $Q'=U^tQU$ and $U^t T U \subseteq T$.
Latter condition is sufficient to guarantee equality $U^t T U = T$.
If $T$ is given by a set of generating quadratic forms or inequalities,
we can easily check computationally if this condition is satisfied.
The same is true for the computation of 
{\em $T$-automorphisms} of~$Q$, 
which are given by all $U\in\gldz$ with $Q=U^tQU$ and $U^t T U \subseteq T$.

In contrast to the classical theory,
finiteness of $T$-perfect forms up to $T$-equivalence may be lost (cf. \cite{js-1994}).  
However, although possibly not finishing in finitely many steps, 
we can generalize Voronoi's algorithm to a graph traversal search
of $T$-equivalent $T$-perfect forms.
Here two $T$-perfect forms are called $T$-contiguous
if they are connected by an edge of the Ryshkov polyhedron $\MNR\cap T$.

In case of $T=T_G$, there exists only 
finitely many $G$-perfect forms
up to scaling and {\em $G$-equivalence}
due to a theorem of Jaquet-Chiffelle \cite{jaquet-1995}.
So in this case we obtain a Voronoi algorithm and
have the possibility to enumerate (in principle)
all $G$-perfect forms up to $G$-equivalence.

In general, we can apply the procedure described in
Algorithm~\ref{alg:voronoi-T-algorithm}
with respect to some given linear subspace~$T$.
If the computation finishes, we have a proof that there
exist only finitely many $T$-inequivalent $T$-perfect forms.

\begin{bigalg}   \label{alg:voronoi-T-algorithm}
\fbox{
\begin{minipage}{12.0cm}
\begin{flushleft}
\smallskip
\textbf{Input:} Dimension $d$ and a linear subspace $T$ of $\sd$.\\
\textbf{Output:} A complete list of $T$-inequivalent $T$-perfect forms in $\sdo\cap T$.\\
\smallskip
Start with a $T$-perfect form $Q$.\\
\begin{enumerate}
\item[1.]
Compute $\Min Q$ and describing inequalities of polyhedral cone
\begin{equation} \label{eqn:Qcone2}
\MP_T(Q)=\{ Q' \in T \; : \; Q'[\vec{x}]\geq 0 \mbox{ for all } \vec{x}\in\Min Q \}
\end{equation}

\item[2.]
Enumerate extreme rays $R_1,\dots, R_k$ of the cone $\MP_T(Q)$

\item[3.]
For indefinite $R_i$, $i=1,\dots,k$,
determine $T$-contiguous $T$-perfect forms
$Q_i=Q+\alpha R_i$

\item[4.]
Test if $Q_i$ is $T$-equivalent to a known form

\item[5.]
Repeat steps 1.--4. for new $T$-perfect forms
\end{enumerate}
\end{flushleft}
\end{minipage}
}
\\[1ex]
{\sc Algorithm \arabic{alg}.} Voronoi's algorithm with respect to a linear subspace $T$.
\end{bigalg}

There are a few differences to Voronoi's Algorithm~\ref{alg:voronoi-algorithm}.
One phenomenon that does not occur in the classical
theory is the possible existence of {\em dead ends}.
These occur at $T$-perfect forms $Q$,
whenever one of the extreme rays $R$ of $\MP_T(Q)$ (as in \eqref{eqn:Qcone2}) 
is positive semidefinite. In this case there is no
$T$-contiguous $T$-perfect form on the ray $\{ Q+\alpha R : \alpha>0\}$.
In fact, the ray is in this case contained in an unbounded face of
the Ryshkov polyhedron $\MNR$.

Another difference to the classical algorithm is that usually 
we do not know a starting $T$-perfect form a priori.
We can however find such a form starting from an initial PQF $Q_0$ in $T$
by applying an adapted version 
of Algorithm~\ref{alg:determination-of-voronoi-neighbors}:
We first compute a maximal linear subspace $L_0$ in $\MP_T(Q_0)$ (as in \eqref{eqn:Qcone2}).
If it is trivial, $Q_0$ is perfect.
Otherwise we choose a form $R$ in $L_0$ which is not positive semidefinite.
We then can
apply Algorithm~\ref{alg:determination-of-voronoi-neighbors}
to $Q=Q_0$ and $R$ and 
obtain a $\rho>0$ such that $Q_1=Q_0 + \rho R$ satisfies
$\lambda(Q_1) = \lambda(Q_0)$ and 
$\Min Q_0 \subset \Min Q_1 \not \subseteq \Min Q_0$.
The maximal linear subspace $L_1$ in $\MP_T(Q_1)$ 
is strictly contained in $L_0$.
By applying this procedure at most $\dim T$ times, we
obtain a $T$-perfect form $Q$.

Note that our viewpoint on $T$-perfect forms
in this article differs from the usual one:
$T$-perfect and $G$-perfect forms 
are usually defined via {\em normal cones}  
of faces of $\MNR \cap T$ in $T$
(cf. \cite{martinet-2003}, \cite{bms-1992},
\cite{jaquet-1995}, \cite{opgenorth-1995} and \cite{opgenorth-2001}).
A face $F$ of $\MNR$ is uniquely characterized by the
set 
$$
\Min F = \{ \vec{x} \in \Z^d : Q[\vec{x}]=1 \mbox{ for all } Q\in F \}
.
$$
The normal cone of $F$ is the Voronoi domain 
$\cone \{ \vec{x}\vec{x}^t : \vec{x} \in \Min F \}$
and the normal cone of the face $F \cap T$ in $T$
is obtained by an orthogonal projection of this Voronoi domain onto $T$.
If different inner products are used, the resulting
cones may differ, as seen in the cases of \cite{jaquet-1995} and
\cite{opgenorth-1995}.

\section{Eisenstein, Gaussian and Hurwitz quaternionic perfect forms}

\label{sec:examples}

As examples for the $G$-theory described 
in the previous section, we consider three cases that 
have been studied intensively before.

{\em Eisenstein forms}.
If $d$ is even, then a $Q\in\sdo$ is said to be an {\em Eisenstein form} 
if it is invariant with respect to a group $G\subset \gldz$ of order~$3$
acting fixed-point-free on $\Z^d\setminus\{\vec{0}\}$ by $z\mapsto Uz$. 
For example 
$$
G=\left\langle \id_{d/2} \otimes \begin{pmatrix} 0 & -1 \\ 1 & -1 \end{pmatrix} \right\rangle
,
$$ 
where $\otimes$ denotes the {\em Kronecker product}.
The terminology comes from the fact that 
a corresponding lattice $L\subset \R^d$ can be viewed as a
{\em complex lattice} of dimension $d/2$ over the {\em Eisenstein integers}
$$
\mathcal{E}= \left\{ a+ b e^{2\pi i/3} : a,b \in \Z \right\}
,
$$
that is, $L=B\mathcal{E}^{d/2}\subset \C^{d/2}$ with a suitable $B\in\GL_{d/2}(\C)$.
On the other hand, it can be seen that each 
complex lattice of this form yields an Eisenstein form.

It turns out that the space of $G$-invariant forms
$T_G$ has dimension $(d/2)^2$.
In particular for $d=2$ we find only one Eisenstein form up to scaling,
associated to the {\em hexagonal lattice} $\mathsf{A}_2$. 
It is trivially $\mathcal{E}$-perfect ({\em Eisenstein perfect}). 
From dimension~$4$ on the situation
is already more interesting. 
In Table~\ref{tab:eisenstein} we list number of classes 
and maximum sphere packing densities of $\mathcal{E}$-perfect 
forms up to dimension~$10$.
Figure~\ref{fig:eisenstein} shows the found
contiguities up to dimension~$8$

\begin{table} 
\begin{tabular}{c|c|c|c|c|c}
$d$  &  $2$  &  $4$  &  $6$  &  $8$  &  $10$  \\
\hline
\hline
$\dim T_{\mathcal{E}}$ & $1$ & $4$ & $9$ & $16$ & $25$ \\
\hline
\# $\mathcal{E}$-perf. forms &  $1$  &  $1$  &  $2$  &  $5$  &  $1628$  \\   \hline 
maximum $\delta$  &  $0.9069\ldots$  &  $0.6168\ldots$  &  $0.3729\ldots$  &  $0.2536\ldots$  &  $0.0360\ldots$  \\
\end{tabular}
\medskip
\caption{Number and maximum densities of $\mathcal{E}$-perfect forms.}
\label{tab:eisenstein}
\end{table}

For $d=4$, the Ryshkov polyhedron is $4$-dimensional in $\sd$
(which itself has dimension $10$).
Up to $\mathcal{E}$-equivalence 
(by mappings $Q\mapsto U^t Q U$ preserving $T_G$),
there is only one $\mathcal{E}$-perfect form, 
namely the one associated to the lattice $\mathsf{D}_4$.
Consequently the Voronoi graph (up to $\mathcal{E}$-equivalence)
is just a single vertex with a loop.
In dimension~$6$, we find already two 
$\mathcal{E}$-inequivalence $\mathcal{E}$-perfect forms,
associated to the lattices $\mathsf{E}_6$ and its dual $\mathsf{E}^{\ast}_6$.

\begin{figure} 
\includegraphics[width=10cm]{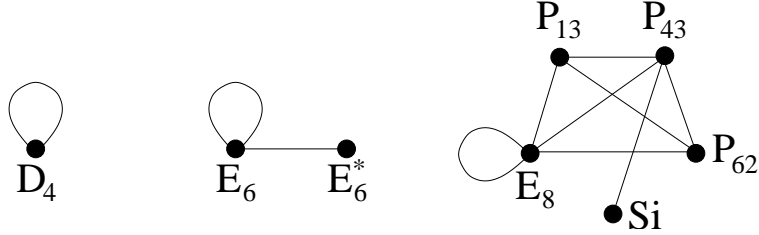}
\caption{Voronoi graphs for $\mathcal{E}$-perfect forms for $d=4,6,8$.}
\label{fig:eisenstein}
\end{figure}

The classification of Eisenstein forms in dimension~$8$ was almost finished
by Sigrist in \cite{sigrist-2004}. 
He found all five classes of $\mathcal{E}$-perfect forms
and their neighboring relations.
However, he could not rule out the existence of other 
{\em $\mathcal{E}$-contiguous} neighbors   
of the forms associated to~$\mathsf{E}_8$.
Recently we finished the classification using a C++-implementation of
the algorithms described in Sections~\ref{sec:voronoi} and \ref{sec:Tperfect}. 
The forms labeled $\mathsf{P}_{13}$, $\mathsf{P}_{43}$ and $\mathsf{P}_{62}$
in Figure~\ref{fig:eisenstein}
are also perfect forms in the classical sense. The index of the labels
corresponds to the number of the class given in the complete list of 
$8$-dimensional perfect forms that can be obtained from our
webpage.\footnote{
see {\footnotesize \url{http://fma2.math.uni-magdeburg.de/~achill/perfect-forms-dim8.txt}}}
The lattice associated to $\mathsf{P}_{62}$ 
is also known as {\em Barnes lattice} $\mathsf{L}_8$
(see \cite[Section 8.4]{martinet-2003}).
The ``Sigrist form'' labeled $\mathsf{Si}$ is an example 
of an $\mathcal{E}$-perfect form 
which is not perfect in the classical sense 
(as already observed in \cite{sigrist-2004}).

Using our implementation we were also able to enumerate all
$10$-dimensional $\mathcal{E}$-perfect forms, showing that their total 
number ``explodes'' to $1628$. 
The data of our classification
can be obtained from our webpage.\footnote{
see {\footnotesize \url{http://fma2.math.uni-magdeburg.de/~achill/E-perfect-forms-dim??.txt}}
where ?? should be replaced by $4$, $6$, $8$ or $10$.} 
The files contain a complete description of the Voronoi graph.

Note that the largest known lattice sphere packing density $\delta$
is attained among $\mathcal{E}$-perfect forms up to dimension~$10$.
A noteworthy phenomenon that occurs among these forms in dimension $10$ is
the existence of $\mathcal{E}$-inequivalent $\mathcal{E}$-perfect forms, 
which are nevertheless arithmetically equivalent. 
This happens for two arithmetically equivalent
forms associated to the lattice $K'_{10}$ (see \cite[Section 8.5]{martinet-2003}).

{\em Gaussian forms}.
For even $d$, a {\em Gaussian form} $Q\in\sdo$ is defined 
as a form containing 
a group $G\subset \gldz$ of order~$4$ in their
automorphism group acting fixed-point-free on $\Z^d\setminus\{\vec{0}\}$.
For example 
$$
G=\left\langle \id_{d/2} \otimes \begin{pmatrix} 0 & -1 \\ 1 & 0 \end{pmatrix} \right\rangle
.
$$ 
A corresponding lattice $L\subset \R^d$ can be viewed as a
complex lattice of dimension $d/2$ over the {\em Gaussian integers}
$$
\mathcal{G}= \left\{ a+ b i : a,b \in \Z \right\}
.
$$
Vice versa, every such lattice yields a Gaussian form.

\begin{figure} 
\includegraphics[width=9cm]{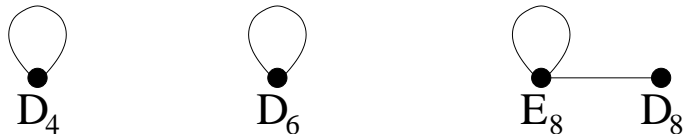}
\caption{Voronoi graphs for $\mathcal{G}$-perfect forms for $d=4,6,8$.} 
\label{fig:gaussian}
\end{figure}

As in the case of Eisenstein forms, 
it turns out that $T_G$ has dimension $(d/2)^2$.
For $d=2$ we find only one $\mathcal{G}$-perfect ({\em Gaussian perfect}) form 
up to scaling, namely $\Z^2$.  
As shown in Figure~\ref{fig:gaussian}, 
the only $\mathcal{G}$-perfect forms in dimension~$6$ and $8$
are associated to the lattices 
$\mathsf{D}_6$, $\mathsf{D}_8$ and $\mathsf{E}_8$.
As shown in Table~\ref{tab:gaussian} the number of equivalence classes 
$\mathcal{G}$-perfect forms 
in dimension~$10$ grows even beyond the corresponding number for
$\mathcal{E}$-perfect forms. 
So far we were not able to finish the classification,
but we think it is computationally within reach on a suitable computer.

\begin{table} 
\begin{tabular}{c|c|c|c|c|c}
$n$  &  $2$  &  $4$  &  $6$  &  $8$  &  $10$  \\
\hline\hline
$\dim T_{\mathcal{G}}$ & $1$ & $4$ & $9$ & $16$ & $25$ \\
\hline
\# $\mathcal{G}$-perf. forms  &  $1$  &  $1$  &  $1$  &  $2$  &  $\geq 17757$  \\
\hline
maximum $\delta$  &  $0.7853\ldots$  &  $0.6168\ldots$  &  $0.3229\ldots$  &  $0.2536\ldots$  & 
\end{tabular}
\medskip
\caption{Number and maximum densities of $\mathcal{G}$-perfect forms.}
\label{tab:gaussian}
\end{table}

As in the case of $\mathcal{E}$-perfect forms, 
the enumeration in dimension~$8$
was started by Sigrist \cite{sigrist-2004}.
However, he did not finish the
classification of {\em $\mathcal{G}$-contiguous} 
$\mathcal{G}$-perfect neighbors 
of $\mathsf{E}_8$.
Nevertheless,
our computations show that his list was nevertheless complete.
The data of our classification
can be obtained from our webpage.\footnote{
see {\footnotesize \url{http://fma2.math.uni-magdeburg.de/~achill/G-perfect-forms-dim??.txt}}
where ?? should be replaced by $4$, $6$ or $8$.} 
Note that in dimensions not divisible by $4$, the
forms giving the densest known lattice sphere packing are not Gaussian.

{\em Hurwitz quaternionic forms}.
For $d$ divisible by $4$, 
a form $Q\in\sdo$ is called {\em Hurwitz quaternionic}
if it is invariant with respect to a group $G\subset \gldz$
isomorphic to $2A_4$ and acting fixed-point-free on 
$\Z^d\setminus\{\vec{0}\}$.
Here, $A_4$ denotes the {\em alternating group} of degree~$4$.
There is a correspondence between Hurwitz quaternionic forms
and lattices in $\R^d$ which can be viewed as
{\em Hurwitz quaternionic lattices} over the 
{\em Hurwitz quaternionic integers}
$$
\mathcal{H}
=
\left\{a + b i + c j + d k : a,b,c,d \in \Z \mbox{ or } a,b,c,d \in \Z+\tfrac{1}{2} \right\}
.
$$
We refer to \cite[Section 2.6]{cs-1998} for details.
%Here $i^2=j^2=k^2=-1$ and $ij=-ji=k, jk=-kj=i, ki=-ik=j$
%are the defining relations of the {\em Quaternion skew field}.
%$\mathbb{H}=\{w + x i + y j + z k : w,x,y,z\in\R\}$.
%An $\mathcal{H}$-lattice is the set of all linear combinations
%$h_1 \vec{a}_1 + \ldots + h_1 \vec{a}_$ ...
%NOTE THAT IT IS ONLY INVARIANT WITH RESPECT TO LEFT MULTIPLICATION!

\begin{figure} 
\includegraphics[width=11cm]{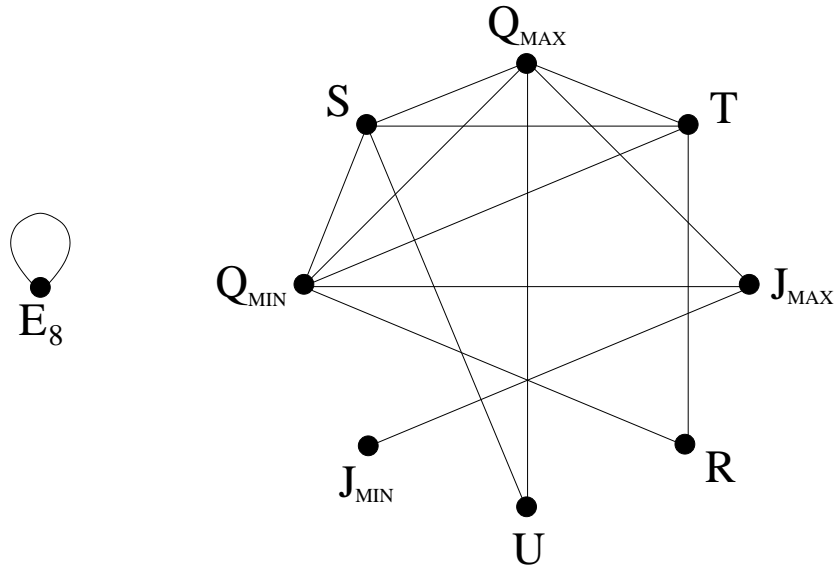}
\caption{Voronoi graphs for  $\mathcal{H}$-perfect forms for $d=8,12$.}
\label{fig:hurwitz}
\end{figure}

It turns out that $T_G$ is of dimension $\binom{d/2}{2}$.
This leaves only one Hurwitz quaternionic form 
and therefore only one {\em $\mathcal{H}$-perfect form}
up to scaling for $d=4$,
which is associated to $\mathsf{D}_4$. 
As shown in Figure~\ref{fig:hurwitz}, there is also only 
one equivalence class of $\mathcal{H}$-perfect forms in dimension~$8$, 
corresponding to $\mathsf{E}_8$.

\begin{table} 
\begin{tabular}{c|c|c|c|c}
$d$  &  $4$  &  $8$  &  $12$  &  $16$  \\
\hline\hline 
$\dim T_{\mathcal{H}}$ & $1$ & $6$ & $15$ & $28$  \\
\hline
\# $\mathcal{H}$-perf. forms  &  1  &  1   &  8  &  ? \\
\hline
maximum $\delta$  &  $0.6168\ldots$  &  $0.2536\ldots$  &
                  $0.03125\ldots$  &  $0.01471\ldots$ \\   
\end{tabular}
\medskip
\caption{Number and maximum densities of $\mathcal{H}$-perfect forms.}
\label{tab:hurwitz}
\end{table}

The situation becomes more interesting in dimension~$12$ 
(cf. Table~\ref{tab:hurwitz}).
By our computations,
there are precisely eight classes of $\mathcal{H}$-perfect forms, 
as previously observed by Jaquet-Chiffelle and 
Sigrist (cf. \cite{sigrist-2008}).
Figure~\ref{fig:hurwitz} uses their labeling.
The data of our computations
can be obtained from our webpage.\footnote{
see
{\footnotesize \url{http://fma2.math.uni-magdeburg.de/~achill/H-perfect-forms-dim??.txt}}
where ?? should be replaced by $8$ or $12$.}
Note that all $\mathcal{H}$-perfect forms are 
also perfect in the classical sense.

A quite interesting consequence of the classification in dimension~$12$
is the possibility to derive of a sharp bound for the largest possible 
sphere packing density among Hurwitz quaternionic forms in 
dimension~$16$, as shown by Vance \cite{vance-2008} using a Mordell type
inequality. She shows that the Barnes-Wall lattice $\mathsf{BW}_{16}$
has the largest density among lattices with a Hurwitz quaternionic structure
in dimension~$16$.
A very nice example of a human-computer-interacted proof!

\section*{Acknowledgements}

The author likes to thank 
Henry Cohn, Rainer Schulze-Pillot, 
Mathieu Dutour Sikiri\'c, Francois Sigrist, 
Jacques Martinet and Stephanie Vance for 
helpful comments and communications.

%
% BIBLIOGRAPHY
%

\providecommand{\bysame}{\leavevmode\hbox to3em{\hrulefill}\thinspace}
\providecommand{\MR}{\relax\ifhmode\unskip\space\fi MR }
\providecommand{\MRhref}[2]{%
  \href{http://www.ams.org/mathscinet-getitem?mr=#1}{#2}
}
\providecommand{\href}[2]{#2}

\end{document}